\documentclass[11pt, leqno]{amsart}
\usepackage{amsfonts,amssymb, amscd}
\usepackage{amsmath}
\usepackage{lmodern}
\usepackage{mathrsfs}
\usepackage{amsthm}
\usepackage{qsymbols}
\usepackage{mathtools}
\mathtoolsset{showonlyrefs}
\usepackage{dsfont}
\usepackage{graphicx}
\usepackage{latexsym}
\usepackage{chngcntr}
\usepackage[noadjust]{cite}
\usepackage{paralist}
\usepackage[parfill]{parskip}
\usepackage{bm}
\usepackage{esint}
\usepackage{tikz-cd}
\usepackage{nicefrac}
\usepackage{float}
\usepackage{csquotes}
\usepackage{euflag}

\usepackage[shortlabels]{enumitem}
\setlist[enumerate,1]{label = \normalfont(\roman*), ref = (\roman*)}
\usepackage{todonotes}
\newtheorem{theorem}{Theorem}[section]
\newtheorem{lemma}[theorem]{Lemma}
\newtheorem{proposition}[theorem]{Proposition}

\newtheorem{taggedtheoremx}{Assumption}
\newenvironment{assumption}[1]
{\renewcommand\thetaggedtheoremx{#1}\taggedtheoremx}
{\endtaggedtheoremx}
\theoremstyle{definition}
\newtheorem{definition}[theorem]{Definition}
\newtheorem{remark}[theorem]{Remark}
\newtheorem{example}[theorem]{Example}

\usepackage[plainpages=false,pdfpagelabels,backref=page,
citecolor=red]{hyperref}
\usepackage{xcolor}
\hypersetup{
	colorlinks,
	linkcolor={cyan!90!black},
	citecolor={magenta},
	urlcolor={green!40!black}
} %

\newcommand{\R}{\mathbb{R}}
\newcommand{\C}{\mathbb{C}}
\newcommand{\N}{\mathbb{N}}
\renewcommand{\L}{\mathrm{L}}

\newcommand{\W}{\mathrm{W}}
\newcommand{\Cont}{\mathrm{C}}

\newcommand{\e}{\mathrm{e}}
\renewcommand{\d}{\,\mathrm{d}}
\newcommand{\ddt}{\,\frac{\mathrm{d}t}{t}}
\newcommand{\dds}{\,\frac{\mathrm{d}s}{s}}
\newcommand{\ddxt}{\,\frac{\mathrm{d}x \mathrm{d}t}{t}}
\let\rr\r
\renewcommand{\r}{\mathrm{r}}
\newcommand{\len}{\ell}
\newcommand{\eps}{\varepsilon}
\newcommand{\B}{\mathrm{B}}
\newcommand{\Q}{\mathrm{Q}}
\renewcommand{\H}{\mathrm{H}}

\renewcommand\Re{\operatorname{Re}}

\newcommand{\ind}{\mathbf{1}}
\let\ii\i

\DeclareMathOperator{\supp}{supp}
\DeclareMathOperator{\dist}{d}
\DeclareMathOperator{\diam}{diam}
\DeclareMathOperator{\dom}{D}
\newcommand{\cL}{\mathcal{L}}

\newcommand{\AvOp}{E}
\DeclareMathOperator{\loc}{loc}
\DeclareMathOperator{\Id}{Id}

\DeclareMathOperator{\Div}{\underline{div}}
\DeclareMathOperator{\Divv}{div}
\newcommand{\wt}{\widetilde}
\newcommand{\Tr}{\gamma}
\newcommand{\I}{\mathrm{I}}
\newcommand{\II}{\mathrm{II}}
\newcommand{\III}{\mathrm{III}}
\DeclareMathOperator{\sgn}{sgn}

\def\Yint#1{\mathchoice
	{\YYint\displaystyle\textstyle{#1}}%
	{\YYint\textstyle\scriptstyle{#1}}%
	{\YYint\scriptstyle\scriptscriptstyle{#1}}%
	{\YYint\scriptscriptstyle\scriptscriptstyle{#1}}%
	\!\iint}
\def\YYint#1#2#3{{\setbox0=\hbox{$#1{#2#3}{\iint}$}
		\vcenter{\hbox{$#2#3$}}\kern-.51\wd0}}
\def\longdash{{-}\mkern-3.5mu{-}}

\def\fiint{\Yint\longdash}

\makeatletter
\@namedef{subjclassname@2020}{%
	\textup{2020} Mathematics Subject Classification}
\makeatother

\title[The Kato square root estimate with Robin boundary conditions]{The Kato square root estimate \\ with Robin boundary conditions}

\author{Sebastian Bechtel}
\address{Université Paris-Saclay, CNRS \\ Laboratoire de Mathématiques d’Orsay \\ 91405 Orsay \\ France}
\email{sebastian.bechtel@universite-paris-saclay.fr}

\author{Andreas Ros\'en}
\address{Mathematical Sciences, Chalmers University of Technology and University of Gothenburg \\ SE-412 96 G{\"o}teborg, Sweden}
\email{andreas.rosen@chalmers.se}

\thanks{Andreas Ros\'en was supported by the Swedish Research Council (Grant 2022-03996), which also funded the short visit by Sebastian Bechtel to Gothenburg during which this work was initiated. This project has received funding from the European Union’s Horizon 2020 research and innovation programme under the Marie Skłodowska-Curie grant agreement No 101034255 \euflag{}.}

\date{\today}
\dedicatory{}

\begin{document}
	\begin{abstract}

		We prove the Kato square root estimate for second-order divergence form elliptic operators $-\Divv(A\nabla)$ on a bounded, locally uniform domain $D \subseteq \mathbb{R}^n$, for accretive coefficients $A \in \L^\infty(D; \mathbb{C}^n)$, under the Robin boundary condition $\nu \cdot A\nabla u + bu = 0$ for a (possibly unbounded) boundary conductivity $b$. In contrast to essentially all previous estimates of Kato square root operators, no first-order approach seems possible for the Robin boundary conditions.

	\end{abstract}
	\maketitle

	\section{Introduction}
	\label{sec:intro}

	This paper concerns estimates for the Kato square root $\sqrt{L}$ of divergence form second-order elliptic operators $L = -\Divv(A\nabla)$. The estimate
	\begin{align}
		\label{eq:intro_kato_estimate}
		\| \sqrt{L} u \|_{\L^2(\R^n)} \approx \| \nabla u \|_{\L^2(\R^n)}
	\end{align}
	conjectured by Kato~\cite{Kat61}, was eventually proved in the seminal work~\cite{Kato-Square-Root-Proof}. Following this breakthrough, there has been many generalizations, for example to domains $D \subseteq \R^n$, imposing Dirichlet, Neumann or mixed such boundary conditions, to manifolds, to degenerate/unbounded coefficients $A$, to more general elliptic systems, and to parabolic operators. See for example~\cite{Morris,AKM,Bandara_McIntosh,ARR,AMR,Parabolic,Parabolic_SecondOrder,Stokes,Cruz_Rios,Brocchi_Rosen,BEH}.

	Although the estimate~\eqref{eq:intro_kato_estimate} seems to concern second-order PDEs, it can be argued that it actually concerns first-order elliptic systems. Indeed,~\eqref{eq:intro_kato_estimate} readily follows from the boundedness of a generalized Cauchy singular integral operator $\sgn(BD)$, where
	\begin{align}
		B = \begin{bmatrix} 1 & 0 \\ 0 & A \end{bmatrix} \quad \text{and} \quad D = \begin{bmatrix} 0 & -\Divv \\ \nabla & 0 \end{bmatrix}.
	\end{align}
	In dimension $n = 1$ this observation goes back to~\cite{AMcN}, and after the Kato square root estimate~\eqref{eq:intro_kato_estimate} was established in full generality in~\cite{Kato-Square-Root-Proof}, the framework from~\cite{AMcN} was generalized in~\cite{AKM-QuadraticEstimates} to arbitrary dimension. See~\cite{AAM-ArkMath} for a simplified presentation of this first-order approach. Although the generalized first-order Cauchy--Riemann system $BD$ looks more complicated than the second-order operator $L$, with essentially the same techniques one can more generally prove bounds
	\begin{align}
		\| f(BD) u \|_{\L^2} \lesssim \| f \|_\infty \| u \|_{\L^2}
	\end{align}
	for all operators $f(BD)$ in the holomorphic functional calculus of $BD$. This has important applications for example to the study of boundary value problems. See~\cite{AAM-ArkMath,AMR,AA-Inventiones}.

	In all previous works on Kato square root estimates, to the best of our knowledge (and possible with the exception of~\cite{Stokes}),
	it has been possible either to do a second-order proof (roughly following the original proof~\cite{Kato-Square-Root-Proof}), or to do a first-order proof (roughly following~\cite{AKM-QuadraticEstimates} or~\cite{AAM-ArkMath}).
	Also for Kato square root estimates on domains, both choices have been available~\cite{AKM,ISEM_Kato,Darmstadt-KatoMixedBoundary}.
	For Dirichlet boundary conditions one uses
	\begin{align}
		D = \begin{bmatrix} 0 & -\Divv \\ \nabla_0 & 0 \end{bmatrix},
	\end{align}
	where $\dom(\nabla_0) = \W^{1,2}_0(D)$ and $-\Divv = \nabla_0^*$. For Neumann boundary conditions, one uses
	\begin{align}
		D = \begin{bmatrix} 0 & -\Divv_0 \\ \nabla & 0 \end{bmatrix},
	\end{align}
	where $\dom(\nabla) = \W^{1,2}(D)$ and $-\Divv_0 = \nabla^*$. If the domain $D$ is sufficiently regular, one finds $\dom(\Divv_0) = \{ f \in \L^2(D) \colon \Divv(f) \in \L^2(D), \nu \cdot f|_{\partial D} = 0 \}$, where $\nu$ denotes the outwards normal to $D$. Even for mixed Dirichlet/Neumann boundary conditions one can define a self-adjoint operator $D$ by putting the Dirichlet boundary conditions on $\nabla$ and the Neumann boundary conditions on $\Divv$.

	In the present paper, we prove the Kato square root estimate on domains $D \subseteq \R^n$ under Robin boundary conditions (formally) given by
	\begin{align}
		\nu \cdot A \nabla u + b u = 0 \quad \text{on } \partial D.
	\end{align}
	We recall that when Fourier's law is used to model heat equilibrium by $\Divv(A\nabla u) = 0$, then $A$ is the conductivity in $D$,
	while $b \geq 0$ is the conductivity across the boundary $\partial D$. The extreme case $b = 0$ means that $D$ is an insulated body and corresponds to Neumann boundary conditions, while the limit $b \to \infty$ means that $u|_{\partial D}$ must equal the outside temperature $0$.

	Unlike for previous works, for the Robin boundary conditions we are forced to do a second-order proof of the Kato square root estimate. Indeed, it is not clear to us even how to define a first-order self-adjoint operator $D$ which is relevant for the Robin boundary conditions. Also our second-order proof here presents new challenges since the Robin form
	\begin{align}
		\label{eq:intro_form}
		a(u,v) =  \iint_D A\nabla u \cdot \overline{\nabla v} \d x + \int_{\partial D} b \gamma(u) \overline{\gamma(v)} \d \sigma
	\end{align}
	contains a boundary term that needs to be converted to a boundary condition for the Robin operator $L$.
	At the technical level, a main novelty in the proof is Lemma~\ref{lem:volume}, which allows us to reduce the Kato square root estimate to a square function estimate for a family of operators acting only on functions in the interior $D$. Our main result is Theorem~\ref{thm:main}, where the Robin boundary conditions $\nu \cdot A\nabla u + bu = 0$ are imposed in a weak sense on a subset $\Gamma \subseteq \partial D$ of dimension $d$, $n-2 < d < n$, of the boundary, and Neumann boundary conditions on the remainder of $\partial D$.

	\section{Geometry and function spaces}
	\label{sec:geometry}

	Fix the dimension $n \geq 2$. We work in an open, bounded and possibly non-connected set $D \subseteq \R^n$. Fix a (relatively) open, non-empty subset $\Gamma \subseteq \partial D$.
	Throughout, we distinguish integrals in the interior and on the boundary by writing $\iint_D \d x$ and $\int_\Gamma \d \sigma$. To save space, we occasionally drop $\d x$ and $\d \sigma$ from the notation.

	We introduce some notions from geometric measure theory in order to ensure some crucial properties for $D$ and $\Gamma$ as well as for function spaces on them.

	\begin{definition}[Locally uniform]
		\label{def:locally_uniform}
		Call $D$ locally uniform, if there exist $\eps \in (0,1]$, $\delta \in (0,\infty]$ and $c \in (0, \infty]$ such that the following hold:
		\begin{enumerate}
			\item[(i)] All points $x, y \in D$ with $|x-y| < \delta$ can be joined in $D$ by an \emph{$\eps$-cigar}, that is to say, a rectifiable curve $c \subseteq D$ of length $\len(c) \leq \eps^{-1} |x-y|$ such that
			\begin{align}
				\dist(z, \partial D) \geq \frac{\eps|z-x||z-y|}{|x-y|}, \quad z \in c.
			\end{align}
			\item[(ii)] All connected components $D'$ of $D$ satisfy $\diam(D') \geq c$.
		\end{enumerate}
	\end{definition}

	\begin{remark}
		When $D$ is connected, the definition coincides with the one of Jones~\cite{Jones}. He has studied properties of Sobolev spaces on $D$, including extension and density results. The non-connected case is a by-product of~\cite{BHT}.
	\end{remark}

	If a set $D \subseteq \R^n$ is locally uniform, then it is in particular interior thick in the sense of the following definition, see~\cite[Prop.~2.9]{BEH}. Interior thickness will be important to ensure standard properties of function spaces on $D$ on the one hand, and to ensure the existence of a dyadic structure on $D$ on the other hand, see Section~\ref{subsec:dyadic} for the latter.

	\begin{definition}[Interior thickness]
		\label{def:ITC}
		Call $O$ interior thick if the following holds true:
		\begin{align}
			\tag{ITC}
			\label{eq:ITC}
			\exists c \in (0,1] \colon \; \forall x \in D, r \in (0,1]\colon \quad |D \cap \B(x,r)| \geq c |\B(x,r)|.
		\end{align}
	\end{definition}

	For the boundary part $\Gamma$, we introduce the following.

	\begin{definition}[$d$-set]
		Call $\Gamma$ a $d$-set for $d \in (0,n]$ if the following holds true:
		\begin{align}
			\label{eq:l_set}
			\exists c \in (0,1] \colon \; \forall x \in \Gamma, r \in (0,1]\colon \quad c r^d \leq \mathcal{H}^d(\Gamma \cap \B(x,r)) \leq c^{-1} r^d.
		\end{align}
	\end{definition}

	Throughout the article, we impose the following geometric setting.

	\begin{assumption}{D}
		\label{ass:D}
		The pair $(D,\Gamma)$ is supposed to satisfy the following:
		\begin{enumerate}
			\item The open set $D$ is supposed to be locally uniform in the sense of Definition~\ref{def:locally_uniform}.
			\item The set $\Gamma$ is an $d$-set for some $d \in (n-2,n)$.
		\end{enumerate}
	\end{assumption}

	It ensures properties of function spaces on $D$ as well as the existence and boundedness of trace operator.

	The space $\W^{1,2}(D)$ used in the formulation of the Kato square root estimate is the usual Sobolev space of $\L^2(D)$ functions whose distributional gradient belongs again to $\L^2(D)$. We introduce further spaces of \emph{Bessel potential} type in the subsequent definition.
	To this end, recall the scale of Bessel potential spaces $\H^{t,r}(\R^n)$ on $\R^n$.

	\begin{definition}
		\label{def:Htp}
		Let $t\in \R$ and $r\in (1,\infty)$. First, if $t \geq 0$ define the space $\H^{t,r}(D)$ by $$\H^{t,r}(D) \coloneqq \{ u|_D \colon u \in \H^{t,r}(\R^n) \}$$ and equip it with the quotient norm. Here, the operation $|_D$ is the distributional restriction $\R^n \to D$. Second, if $t < 0$, define the space $\H^{t,r}(D)$ as the topological anti-dual space\footnote{That is, the space of conjugate-linear bounded functionals.} of $\H^{-t, r'}(D)$.
	\end{definition}

	By construction, $\H^{t,r}(D)$ is a Banach space (Hilbertian if $r=2$), $\H^{t,r}(D) \cap \Cont(\overline{D})$ is dense in $\H^{t,r}(D)$ for all $t \in \R$ and $r\in (1,\infty)$, and the usual Sobolev embeddings apply.
	If $t_1 > t_0$, then $\H^{t_1,r}(D) \subseteq \H^{t_0,r}(D)$ with dense inclusion.
	Moreover, as $D$ satisfies the interior thickness condition, the spaces $\H^{t,r}(D)$ obey the same interpolation rules as in the case $D = \R^n$, see~\cite[Prop.~3.8]{BE}.

	Since $D$ is locally uniform, it is a Sobolev extension domain~\cite{Jones,BHT}. This leads to compatibility between the $\W$ and $\H$ spaces.

	\begin{lemma}
		\label{lem:H=W}
		For $r \in (1,\infty)$, the spaces $\W^{1,r}(D)$ and $\H^{1,r}(D)$ coincide with equivalence of norms.
	\end{lemma}

	\begin{proof}
		Recall $\H^{1,r}(\R^n) = \W^{1,r}(\R^n)$.
		First, let $u \in \W^{1,r}(D)$ and write $\mathcal{E}$ for the bounded extension operator $\W^{1,r}(D) \to \W^{1,r}(\R^n)$ introduced in~\cite{BHT}. Then
		\begin{align}
			\| u \|_{\H^{1,r}(D)} \leq \| \mathcal{E} u \|_{\H^{1,r}(\R^n)} = \| \mathcal{E} u \|_{\W^{1,r}(\R^n)} \lesssim \| u \|_{\W^{1,r}(D)},
		\end{align}
		which gives the inclusion $\W^{1,r}(D) \subseteq \H^{1,r}(D)$. Second, let $u \in \H^{1,r}(D)$ and let $U \in \H^{1,r}(\R^n)$ be any function with $U|_D = u$. Calculate
		\begin{align}
			\| u \|_{\W^{1,r}(D)} \leq \| U \|_{\W^{1,r}(\R^n)} = \| U \|_{\H^{1,r}(\R^n)}.
		\end{align}
		Taking the infimum over all extensions $U$ of $u$ yields $\| u \|_{\W^{1,r}(D)} \leq \| u \|_{\H^{1,r}(D)}$, which gives the converse inclusion $\H^{1,r}(D) \subseteq \W^{1,r}(D)$.
	\end{proof}

	As for the boundary, for $r \in [1,\infty]$ put $\L^r(\Gamma) \coloneqq \L^r(\Gamma, \sigma)$, where $\sigma$ denotes the $d$-dimensional Hausdorff measure. Note that $\Gamma$ is has positive and finite $\sigma$-measure because it is an non-empty bounded $d$-set.

	Boundedness of the trace operator in our geometric setting was investigated in~\cite[Sec.~2]{BER}. The following result follows by an argument in the spirit of~\cite[Thm.~2.13 \& Thm.~2.16]{BER} the details are left to the reader.

	\begin{proposition}[Trace operator]
		\label{prop:trace}
		Let $r \in [2, \tfrac{2d}{n-2})$ and $s \in (\tfrac{n}{2} - \tfrac{d}{r}, 1]$.
		Then, the trace operator $\Tr$ onto $\Gamma$, originally defined on $\H^{s,2}(D) \cap \Cont(\overline{D})$, extends by density to a bounded operator $\H^{s,2}(D) \to \L^r(\Gamma)$.
	\end{proposition}

	Note that, owing to Lemma~\ref{lem:H=W}, $\W^{1,2}(D) \subseteq \H^{s,2}(D)$ continuously.

	\section{Elliptic operators with Robin boundary conditions}
	\label{sec:elliptic}

	We define elliptic operators with Robin boundary conditions and collect some important properties of them.
	Also, we state our main result in Theorem~\ref{thm:main}.
	Let $A \in \L^\infty(D; \C^{n\times n})$ and $b \in \L^{\tilde q}(\Gamma)$.
	It is possible to include lower-order terms, but we stick to the pure second-order case for the sake of readability.
	Note that, using $d > n-2$ and $n \geq 2$, the denominator $d + 2 - n$ in the definition of $\tilde q$ is strictly positive and $\tilde q > 1$ .

	The following lemma relates $\tilde q$ to boundedness properties of the trace operator.

	\begin{lemma}
		\label{lem:parameter}
		There are $q \in [2,\infty)$ and $s_*\in (0,1)$ satisfying
		\begin{align}
			1 = \frac{1}{\tilde q} + \frac{2}{q}, \qquad s_* = \frac{n}{2} - \frac{d}{q},
		\end{align}
		for which the trace operator $\Tr$ from Proposition~\ref{prop:trace} is $\H^{s,2}(D) \to \L^q(\Gamma)$ bounded for all $s \in (s_*, 1]$.
	\end{lemma}

	\begin{proof}
		Define $q$ and $s_*$ via the relations in the statement.
		Since
		\begin{align}
			\label{eq:parameters}
			\frac{n}{2} - \frac{d}{q} = \frac{n}{2} - \frac{d}{2}\Bigl(1-\frac{1}{\tilde q}\Bigr) < \frac{n}{2} - \frac{d}{2} \frac{n-2}{d} = 1
		\end{align}
		on the one hand, and
		\begin{align}
			\frac{n}{2} - \frac{d}{q} = \frac{n}{2} - \frac{d}{2} + \frac{d}{2\tilde q} \geq \frac{n-d}{2} > 0
		\end{align}
		on the other hand. Hence, $s_* \in (0,1)$.
		For the boundedness of the trace operator, we apply Proposition~\ref{prop:trace} with $r \coloneqq q$ and $s \in (s_*, 1]$.
		Indeed, $s$ is admissible by definition, $q \geq 2$ follows from $\tilde q \leq \infty$, and we have $\frac{1}{q} > \frac{n-2}{2d}$ by rearranging the terms in~\eqref{eq:parameters}, which concludes the proof.
	\end{proof}

	Throughout, fix some $s\in (s_*, 1)$. Note that the interval $(s_*, 1)$ is non-empty by virtue of Lemma~\ref{lem:parameter}.

	Define the sesquilinear form $a$ on $\W^{1,2}(D)$ by
	\begin{align}
		\label{eq:a}
		\tag{a}
		a(u,v) \coloneqq \iint_D A \nabla u \cdot \overline{\nabla v} \d x + \int_\Gamma b \Tr(u) \cdot \overline{\Tr(v)} \d \sigma \qquad (u,v \in \W^{1,2}(D)).
	\end{align}
	Recall that $\sigma$ denotes the $d$-dimensional Hausdorff measure.
	\begin{lemma}
		The form $a \colon \W^{1,2}(D) \times \W^{1,2}(D) \to \C$ is bounded.
	\end{lemma}
	\begin{proof}
		It suffices to consider the boundary term. In accordance with Lemma~\ref{lem:parameter}, we use Hölder's inequality and boundedness of the trace operator, to get
		\begin{align}
			\Bigl| \int_\Gamma b \Tr(u) \cdot \overline{\Tr(v)} \d \sigma \Bigr| \leq \| b \|_{\L^{\tilde q}(\Gamma)} \| \Tr(u) \|_{\L^q(\Gamma)} \| \Tr(v) \|_{\L^q(\Gamma)} \lesssim \| b \|_{\L^{\tilde q}(\Gamma)} \| u \|_{\H^{s,2}(D)} \| u \|_{\H^{s,2}(D)}.
		\end{align}
		By virtue of the embedding $\W^{1,2}(D) \subseteq \H^{s,2}(D)$, this completes the proof.
	\end{proof}
	To ensure ellipticity, we assume for some $\lambda > 0$ the inhomogeneous G\rr{a}rding inequality
	\begin{align}
		\label{eq:garding}
		\tag{E}
		\Re a(u,u) \geq \lambda \| u \|_{\W^{1,2}(D)}^2, \qquad u \in \W^{1,2}(D).
	\end{align}
	\begin{example}
		If $A \colon D \to \C^{n\times n}$ are elliptic coefficients and $b \colon \Gamma \to (0,\infty)$, then~\eqref{eq:garding} holds. The argument is similar to Step~1 of the proof of Proposition~\ref{prop:coercivity}. Details are left to the reader.
	\end{example}

	We define first a weak formulation of an elliptic operator with Robin boundary conditions.

	\begin{definition}[Elliptic operator with Robin BC -- weak formulation]
		Define the operator $\cL \colon \W^{1,2}(D) \to (\W^{1,2}(D))^*$ by the relation
		\begin{align}
			\langle \cL u, v \rangle = a(u,v), \qquad u,v \in \W^{1,2}(D).
		\end{align}
	\end{definition}
	Note that $\cL$ is invertible by virtue of~\eqref{eq:garding}.

	Rewriting the Hölder relation in Lemma~\ref{lem:parameter} gives $\tfrac{1}{\tilde q} + \tfrac{1}{q} = \tfrac{1}{q'}$. Keeping this in mind, introduce the generalized gradient operators $S$ by
	\begin{align}
		\label{eq:S}
		S &\coloneqq \begin{bmatrix} \nabla \\ \Tr \end{bmatrix} \colon \W^{1,2}(D) \to \L^2(D)^n \oplus \L^q(\Gamma),
	\end{align}
	and the coefficient matrix $B$ by
	\begin{align}
		\label{eq:B}
		B \coloneqq \begin{bmatrix} A & \\ & b \end{bmatrix} \colon \L^2(D)^n \oplus \L^q(\Gamma) \to \L^2(D)^n \oplus \L^{q'}(\Gamma).
	\end{align}
	In terms of $S$ and $B$, the form $a$ can be rewritten as
	\begin{align}
		a(u,v) = \langle B Su, Sv \rangle.
	\end{align}
	The pairing $\langle \cdot, \cdot \rangle$ denotes the direct sum of the $\L^2$ inner products on $D$ or $\Gamma$ for the respective components.
	Consequently, we find the factorization identity
	\begin{align}
		\label{eq:factorization}
		\cL = S^* B S.
	\end{align}

	To simplify notation, put $-\Div \coloneqq (\nabla)^*$, where $\nabla \colon \W^{1,2}(D) \to \L^2(D)^n$ is the usual distributional gradient. The operator $-\Div$ acts as a bounded operator $\L^2(D)^n \to (\W^{1,2}(D))^*$.
	With this notation at hand, we can write
	\begin{align}
		S^* = \begin{bmatrix} -\Div & \gamma^* \end{bmatrix}.
	\end{align}

	\begin{example}
		\label{ex:S*}
		Let $F \colon D \to \C^n$, $g\colon \Gamma \to \C$ and $u\colon D \to \C$. Assume that $F,g,u$ and $D$ are smooth. Then
		\begin{align}
			\langle S^*(F,g), u \rangle = \iint_D \nabla F \cdot \overline{\nabla u} \d x + \int_{\partial D} (\nu \cdot F + \ind_\Gamma g) \overline{\gamma(u)} \d \sigma.
		\end{align}
		Hence, the range of $S^*$ contains distributions in $D$ and on $\partial D$. On $\Gamma$, the contributions from $\nu \cdot F$ and $g$ interact.
	\end{example}
	Now, we restrict $\cL$ to a closed, unbounded operator in $\L^2(D)$.

	\begin{definition}[Elliptic operator with Robin BC -- strong formulation]
		\label{def:L}
		Define the unbounded operator $L$ in $\L^2(D)$ as follows: if $u \in \W^{1,2}(D)$, $f \in \L^2(D)$, then
		\begin{align}
			u \in \dom(L), \; L u = f \quad \Longleftrightarrow \quad \forall v \in \W^{1,2}(D) \colon \; \langle \cL u, v \rangle = \iint_D f \overline{v} \d x.
		\end{align}
	\end{definition}

	As is illustrated by Example~\ref{ex:S*}, the operator $L$ satisfies Robin boundary conditions on $\Gamma$ and Neumann boundary conditions on $\partial D \setminus \Gamma$ in a weak sense.

	\begin{remark}
		\label{rem:weak_Laplacian}
		The choice $A = \Id$ and $b = 1$ leads to a Laplacian $-\Delta$ with certain Robin boundary conditions.
		In its weak formulation, it is simply given by $S^* S$. We will use this operator as an unperturbed smoothing operator for $\cL$ later on.
		Since $S^*S$ is self-adjoint, it satisfies the Kato square root property.
	\end{remark}

	By classical form theory, $L$ is closed, densely defined, maximal accretive and $\omega$-sectorial for some angle $\omega \in [0, \tfrac{\pi}{2})$. In particular, it possesses a unique maximal accretive square root $\sqrt{L}$, which is the operator appearing in Theorem~\ref{thm:main} below. The operator $L^*$ is associated with the adjoint form $a^*(u,v) \coloneqq \overline{a(v,u)}$. Note that $L$ is self-adjoint if and only if $A$ is self-adjoint and $a$ is real.

	\begin{theorem}[Kato square root property for Robin BC]
		\label{thm:main}
		Let $D \subseteq \R^n$ be open and bounded, and let $\Gamma \subseteq \partial D$. Suppose that the pair $(D,\Gamma)$ satisfies Assumption~\ref{ass:D}.
		Let $A = (a_{ij})_{ij} \in \L^\infty(D)$ and $b \in L^{\tilde q}(D)$, where $\tilde q \in (\tfrac{d}{d + 2 - n}, \infty]$ using $d$ from Assumption~\ref{ass:D}.
		Assume that the Robin form $a$ defined in~\eqref{eq:a} is elliptic in the sense of~\eqref{eq:garding},
		and associate an operator $L$ on $\L^2(D)$ with it (Definition~\ref{def:L}).
		Then we have the identity $\dom(\sqrt{L}) = \W^{1,2}(D)$ along with the estimate
		\begin{align}
			\label{eq:kato_estimate}
			\| \sqrt{L} u \|_{\L^2(D)} \approx \| u \|_{\W^{1,2}(D)} \approx \sqrt{\Re a(u,u)}, \quad u \in \W^{1,2}(D),
		\end{align}
		where the implicit constants depend on the coefficients of $L$ only via $\| A \|_{\infty}$, $\| b \|_{\tilde q}$ and $\lambda$.
	\end{theorem}

	\subsection{Off-diagonal estimates}
	\label{subsec:ODE}

	We establish off-diagonal estimates for elliptic operators with Robin boundary conditions. Since $\Tr(u \phi) = \Tr(u) \Tr(\phi)$ for suitable functions $u$ and $\phi$, the proof is similar to the proof for Neumann boundary conditions.

	\begin{proposition}[Off-diagonal estimates]
		\label{prop:ODE}
		There exist constants $C$ and $c$ depending on $\lambda$, $\| A \|_\infty$ and the dimension, such that for all measurable sets $E, F \subseteq D$, all $t > 0$ and all $f \in \L^2(D)$ with $\supp(f) \subseteq E$, we have for $u \coloneqq (1 + t^2 L)^{-1} f$ the estimate
		\begin{align}
			\left( \iint_F |u|^2 \right)^\frac{1}{2} + \left( \int_{\overline{F} \cap \Gamma} |t \Tr(u)|^q \right)^\frac{1}{q} + \left( \iint_F |t u|^2 \right)^\frac{1}{2} + \left( \iint_F |t \nabla u|^2 \right)^\frac{1}{2} \\
			\leq{} C \e^{-c\frac{\dist(E,F)}{t}} \left( \iint_E |f|^2 \right)^\frac{1}{2}.
		\end{align}
	\end{proposition}

	\begin{proof}
		Fix $t > 0$, let $E, F \subseteq D$ measurable, and put $d \coloneqq \dist(E,F)$.
		Introduce a threshold $\alpha > 0$ to be determined later. We divide the proof into the cases $d \leq \alpha t$ and $d > \alpha t$. Note that the first case simply states boundedness, whereas the second case establishes genuine off-diagonal decay.

		\textbf{Step 1}: boundedness. Note that this step takes in particular care of the case $d \leq \alpha t$. Let $f \in \L^2(D)$ and set $u \coloneqq (1+t^2 L)^{-1} f \in \W^{1,2}(D)$. Using the G\rr{a}rding inequality~\eqref{eq:garding} and maximal accretivity, calculate
		\begin{align}
			\lambda \| t u \|_{\W^{1,2}(D)}^2 \leq t^2 \Re a(u,u) = \Re \langle t^2 L u, u \rangle_{\L^2(D)} \leq \| t^2 L u \|_{\L^2(D)} \| u \|_{\L^2(D)} \leq 2 \| f \|_{\L^2(D)}^2.
		\end{align}
		Taking the square root yields $\| t u \|_{\W^{1,2}(D)} \lesssim \| f \|_{\L^2(D)}$.
		In conjunction with boundedness of the trace operator and maximal accretivity, we deduce
		\begin{align}
			\| u \|_{\L^2(D)} + \| t \Tr(u) \|_{\L^q(\Gamma)} + \| t u \|_{\L^2(D)} + \| t \nabla u \|_{\L^2(D)} \lesssim \| f \|_{\L^2(D)},
		\end{align}
		which concludes this step.

		\textbf{Step 2}: off-diagonal bounds.
		Assume $d > \alpha t$, otherwise nothing has to be shown.
		Let $f \in \L^2(D)$ with $\supp(f) \subseteq E$.
		Note that $E$ and $F$ are disjoint owing to $d > \alpha t > 0$.
		Moreover, let $\rho \in \Cont^\infty(D)$ satisfy $\rho = 1$ on $F$, $\rho = 0$ on $E$, and
		\begin{align}
			\| \rho \|_\infty + d \| \nabla \rho \|_\infty \leq C,
		\end{align}
		where $C$ is a dimensional constant.
		Put $u \coloneqq (1 + t^2 L)^{-1} f$, so that $(1 + t^2 L)u = f$.
		Test this equation with $v \coloneqq u \eta^2$, where $\eta \coloneqq \e^{\frac{d\rho}{\alpha t}} - 1$, to find
		\begin{align}
			\label{eq:ode_testing}
			\iint_D f \overline{v} = \iint_D u \overline{v} + t^2 a(u,v).
		\end{align}
		Note that indeed $v \in \W^{1,2}(D)$, which uses that $\eta$ is bounded and smooth with bounded derivatives for fixed $t$. Also, taking $d > \alpha t$ into account, observe the following properties of $\eta$:
		\begin{align}
			\label{eq:eta_props}
			\eta = 0 \;\; \text{on } E, \quad
			\eta = \e^\frac{d}{\alpha t} - 1 \geq \frac{1}{2} \e^\frac{d}{\alpha t} \;\; \text{on } \overline{F}, \quad
			|\nabla \eta| = |\frac{d}{\alpha t} \nabla \rho \cdot \e^\frac{d\rho}{\alpha t}| \leq \frac{C}{\alpha t} |\eta + 1|.
		\end{align}
		Using the identities $\nabla v = \nabla(u\eta) \eta + u\eta \nabla \eta$ and $\eta \nabla u = \nabla (u\eta) - u\nabla \eta$, which readily follow from the product rule, we find
		\begin{align}
			\label{eq:ode_expand}
				\hspace*{1em}
				a(u,v) - a(\eta u, \eta u)= &- \iint_D u A\nabla \eta \cdot \overline{\nabla (\eta u)} + \iint_D \overline{u} A \nabla (\eta u) \cdot \nabla \eta - \iint_D |u|^2 A\nabla \eta \cdot \nabla \eta.
		\end{align}
		Note that the boundary term does not produce any error terms owing to the observation mentioned just before the proposition.
		Due to the support properties of $f$ and $\eta$, the left-hand side of~\eqref{eq:ode_testing} vanishes. Hence, by using~\eqref{eq:ode_expand}, the identity $u\overline{v} = |\eta u|^2$ and rearranging terms, we see
		\begin{align}
			&\iint_D |\eta u|^2 + t^2 a(\eta u, \eta u) \\
			={} &t^2 \iint_D u A\nabla \eta \cdot \overline{\nabla (\eta u)} - t^2 \iint_D \overline{u} A \nabla (\eta u) \cdot \nabla \eta + t^2 \iint_D |u|^2 A\nabla \eta \cdot \nabla \eta.
		\end{align}
		Then, boundedness of the trace operator (with constant $C_\Gamma > 0$), followed by the G\rr{a}rding inequality~\eqref{eq:garding}, yield
		\begin{align}
			\label{eq:ode_garding}
			\begin{aligned}
				&\iint_D |\eta u|^2 + \frac{\lambda}{2 C_\Gamma^2} t^2 \left( \iint_D |\Tr(\eta u)|^q \right)^\frac{2}{q} + \frac{\lambda}{2} t^2 \iint_D |\eta u|^2 + |\nabla(\eta u)|^2 \\
				\leq{} &\iint_D |\eta u|^2 + \lambda t^2 \iint_D |\eta u|^2 + |\nabla(\eta u)|^2 \\
				\leq{} &\iint_D |\eta u|^2 + t^2 \Re a(\eta u, \eta u) \\
				\leq{} &2 \Lambda t^2 \iint_D |u| |\nabla \eta| |\nabla (\eta u)| + \Lambda t^2 \iint_D |u|^2 |\nabla \eta|^2,
			\end{aligned}
		\end{align}
		where $\Lambda \coloneqq \| A \|_\infty$.
		We estimate the terms on the right-hand side of~\eqref{eq:ode_garding}.
		To start with, use~\eqref{eq:eta_props} to calculate
		\begin{align}
			\Lambda t^2 \iint_D |u|^2 |\nabla \eta|^2 &\leq  \frac{\Lambda C^2}{\alpha^2} \iint_D |u|^2 |\eta + 1|^2
			\leq \frac{2 \Lambda C^2}{\alpha^2} \iint_D |\eta u|^2 + \frac{2 \Lambda C^2}{\alpha^2} \iint_D |u|^2.
		\end{align}
		By taking $\alpha$ large depending on $\Lambda$ and the dimension, the first-term on the right-hand side can be absorbed into the left-hand side of~\eqref{eq:ode_garding}.
		Using this bound again, Young's inequality $2ab \leq \eps a^2 + \tfrac{1}{\eps} b^2$ reveals
		\begin{align}
			&2 \Lambda t^2 \iint_D |u| |\nabla \eta| \nabla (\eta u)| \\
			\leq{} &\frac{\lambda}{8} t^2 \iint_D |\nabla (\eta u)|^2 + \frac{8\Lambda^2}{\lambda} t^2 \iint_D |u|^2 |\nabla \eta|^2 \\
			\leq{} &\frac{\lambda}{8} t^2 \iint_D |\nabla (\eta u)|^2 + \frac{16 \Lambda^2 C^2}{\lambda \alpha^2} \iint_D |\eta u|^2 + \frac{16 \Lambda^2 C^2}{\lambda \alpha^2} \iint_D |u|^2.
		\end{align}
		By taking $\alpha$ large depending on $\lambda$, $\Lambda$ and the dimension, the first two terms on the right-hand side can be absorbed into the left-hand side of~\eqref{eq:ode_garding}.
		In summary, we have shown so far
		\begin{align}
			\label{eq:ODE_penultimate}
			\left( \iint_D |\eta u|^2 \right)^\frac{1}{2} + \left( \int_\Gamma |t \Tr(\eta u)|^q \right)^\frac{1}{q} + \left( \iint_D |t\eta u|^2 + |t\nabla(\eta u)|^2 \right)^\frac{1}{2} \lesssim \left( \iint_D |u|^2 \right)^\frac{1}{2}.
		\end{align}
		Note that $\eta$ is constant on $\overline{F}$, so that $\nabla(\eta u) = \eta \nabla u$ and $\Tr(\eta u) = \eta \Tr(u)$. Then, by~\eqref{eq:eta_props}, we obtain the bound
		\begin{align}
			\frac{1}{2} \e^\frac{d}{\alpha t} \left( \iint_F |u|^2 \right)^\frac{1}{2} + \frac{1}{2} \e^\frac{d}{\alpha t} \left( \int_{\overline{F} \cap \Gamma} |t \Tr(u)|^q \right)^\frac{1}{q} + \frac{1}{2} \e^\frac{d}{\alpha t} \left( \iint_F |t u|^2 + |t \nabla u)|^2 \right)^\frac{1}{2}
			\lesssim \left( \iint_D |u|^2 \right)^\frac{1}{2}.
		\end{align}
		Now, the claim follows from Step~1 upon multiplication by $\e^{-\frac{d}{\alpha t}}$.
	\end{proof}

	\section{A coercivity property for the Laplacian with Robin BC}
	\label{sec:coercivity}

	In this section, we show a coercivity property for $S^*S$. More precisely, we establish coercivity of fractional powers of $S^*S$ of order slightly above $\nicefrac{1}{2}$ in the scale of Bessel potential spaces on $D$.

	\begin{proposition}[Coercivity of $S^*S$]
		\label{prop:coercivity}
		There exists $\alpha \in (0,1]$ such that $\dom((S^*S)^\frac{1+\alpha}{2}) = \H^{1+\alpha,2}(D)$.
		In particular, there exists $c>0$ such that, for every $v \in \dom((S^*S)^\frac{1+\alpha}{2})$, we have the estimate
		\begin{align}
			\label{eq:coercivity_estimate}
			\| (S^*S)^\frac{1+\alpha}{2} v \|_{\L^2(D)} \geq c \| v \|_{\H^{1+\alpha,2}(D)}.
		\end{align}
	\end{proposition}

	\begin{remark}
		\label{rem:coercivity}
		Only the coercivity estimate~\eqref{eq:coercivity_estimate} is needed to establish the Kato square root estimate.
	\end{remark}

	Recall that $\H^{1,2}(D) = \W^{1,2}(D)$. We start out with a technical lemma.

	\begin{lemma}
		\label{lem:extrapolation_laplace}
		There exists $t_0 > 0$ such that, if $0 < t < t_0$, then $S^*S \colon \W^{1,2}(D) \to (\W^{1,2}(D))^*$ restricts/extends to a bounded operator $\H^{1 \pm t,2}(D) \to (\H^{1 \mp t,2}(D))^*$.
	\end{lemma}

	\begin{proof}
		By definition, we have
		\begin{align}
			\langle S^*S u, v \rangle = \iint_D \nabla u \cdot \overline{\nabla v} \d x + \int_\Gamma \Tr(u) \overline{\Tr(v)} \d \sigma, \quad u,v \in \W^{1,2}(D).
		\end{align}
		Assume in addition $u \in \H^{1 \pm t,2}(D)$ and $v \in \H^{1 \mp t,2}(D)$. Since the $\H^{\pm t,2}(D)$ -- $\H^{\mp t,2}(D)$ duality extends the inner product on $\L^2(D)$, boundedness of the gradient in the scale of Bessel potential spaces on $D$ (established in our geometric setting for $t_0$ small enough in~\cite[Lem.~4.10]{BEH}) yields
		\begin{align}
			\Bigl| \iint_D \nabla u \cdot \overline{\nabla v} \d x \Bigr| \leq \| \nabla u \|_{\H^{\pm t,2}(D)} \| \nabla v \|_{\H^{\mp t,2}(D)} \lesssim \| u \|_{\H^{1 \pm t,2}(D)} \| v \|_{\H^{1 \mp t,2}(D)}.
		\end{align}
		Arguing similarly in the boundary, this time using boundedness of the trace operator,
		\begin{align}
			\Bigl| \int_\Gamma \Tr(u) \overline{\Tr(v)} \d \sigma \Bigr| \leq
			\| \Tr(u) \|_{\L^2(\Gamma)} \| \Tr(v) \|_{\L^2(\Gamma)} \lesssim \| u \|_{\H^{1 \pm t,2}(D)} \| v \|_{\H^{1 \mp t,2}(D)},
		\end{align}
		provided $t_0$ is small enough so that $1 - t > \tfrac{n}{2} - \tfrac{d}{2}$. Here, we used that $d > n-2$.
		Therefore, the claim follows by density.
	\end{proof}

	Now, we can proceed with the proof of the coercivity property.

	\begin{proof}[Proof of Proposition~\ref{prop:coercivity}]
		Let $t_0$ be as in Lemma~\ref{lem:extrapolation_laplace}.
		The proof divides onto four steps.

		\textbf{Step 1}: there exists $c > 0$ such that $\langle S^*S u, u \rangle \ge c \| u \|_{\H^1(D)}^2$.

		Note that, since $\langle S^*S u, u \rangle = \| \nabla u \|_{\L^2(D)^n}^2 + \| \Tr(u) \|_{\L^2(\Gamma)}^2$, it suffices to show $\| \nabla u \|_{\L^2(D)^n} + \| \Tr(u) \|_{\L^2(\Gamma)} \ge c \| u \|_{\L^2(D)}$.
		To this end, write $u = (u - (u)_{\Gamma}) + (u)_{\Gamma}$, where $(u)_{\Gamma} \coloneqq \fint_\Gamma \Tr(u) \d \sigma$ is the mean of $\Tr(u)$ over $\Gamma$.
		Here, we use that $\sigma(\Gamma) > 0$ as $\Gamma$ is a $d$-set.
		First, we have
		\begin{align}
			\| (u)_\Gamma \|_{\L^2(D)} = |(u)_\Gamma| |D|^\frac{1}{2} \leq \left( \frac{|D|}{\sigma(\Gamma)} \right)^\frac{1}{2} \| \Tr(u) \|_{\L^2(\Gamma)}
		\end{align}
		by Jensen's inequality.
		Second, consider the kernel $N$ of the bounded functional $\W^{1,2}(D) \ni v \mapsto (v)_\Gamma \in \C$, which is a closed subspace of $\W^{1,2}(D)$ containing no non-trivial constants. Since the embedding $\W^{1,2}(D) \subseteq \L^2(D)$ is compact by boundedness of $D$, the following Poincaré inequality holds:
		\begin{align}
			\| v \|_{\L^2(D)} \lesssim \| \nabla v \|_{\L^2(D)^n}, \quad v \in N.
		\end{align}
		As averaging is a projection, we find $u - (u)_\Gamma \in N$.
		Consequently, we get
		\begin{align}
			\| u - (u)_\Gamma \|_{\L^2(D)} \lesssim \| \nabla(u - (u)_\Gamma) \|_{\L^2(D)^n} = \| \nabla u \|_{\L^2(D)^n}.
		\end{align}
		Combining both bounds completes this step.

		\textbf{Step 2}: there exists $\eps \in (0,t_0)$ such that $S^*S$ is $\H^{1 \pm t,2}(D) \to (\H^{1 \mp t,2}(D))^*$ invertible for all $t \in (0, \eps)$.

		By Step~1 and the Lax--Milgram lemma, $S^*S \colon \H^{1,2}(D) \to (\H^{1,2}(D))^*$ is invertible. Moreover, $S^*S$ is bounded between the complex interpolation scales $\{ \H^{1 + t,2}(D) \}_{-t_0 < t < t_0}$ and $\{ (\H^{1 - t,2}(D))^* \}_{-t_0 < t < t_0}$. Hence, the claim of this step follows from Sneiberg's lemma~\cite{Sneiberg-Original}.

		\textbf{Step 3}: there holds the identity $\dom((S^*S)^\frac{1+t}{2}) = \H^{1+t,2}(D)$ for all $t \in (0,\eps)$.

		We sketch the argument, further details are presented in~\cite[Proof of Thm.~1.2]{BEH}.
		The idea is to decompose
		\begin{align}
			(S^*S)^{-\frac{1+t}{2}} = (S^*S)^{-1} (S^*S)^\frac{1-t}{2}.
		\end{align}
		The left-hand side is an isomorphism onto $\dom((S^*S)^{\frac{1+t}{2}})$.
		Hence, we need to show that the right-hand side is an isomorphism onto $\H^{1+t,2}(D)$.
		We have just established in Lemma~\ref{lem:extrapolation_laplace} that $(S^*S)^{-1}$ is an $\H^{-1+t,2}(D) \to \H^{1+t,2}(D)$ isomorphism.
		As for $(S^*S)^\frac{1-t}{2}$, it follows from the self-adjoint square root property, complex interpolation and a duality argument that it is an $\L^2(D) \to \H^{-1+t,2}(D)$ isomorphism as desired.
		Note that these are generic arguments independent of the Robin boundary conditions.
		This concludes the proof.

		\textbf{Step 4}: conclusion.

		Set $\alpha = \nicefrac{\eps}{2}$. By Step~3, it only remains to establish the coercivity estimate~\eqref{eq:coercivity_estimate}. By invertibility of $S^*S$, the space $\dom((S^*S)^\frac{1+\alpha}{2})$ can equivalently be equipped with the homogeneous graph norm. Then,~\eqref{eq:coercivity_estimate} is immediate.
	\end{proof}

	\section{Reduction to a Carleson measure estimate}
	\label{sec:reduction}

	In a series of reduction steps, we first reformulate the validity of Theorem~\ref{thm:main} in terms of quadratic estimates, which are successively reduced to showing that a certain measure is a Carleson measure. The Carleson measure property is shown in Section~\ref{sec:carleson}.

	\subsection{Reduction to a quadratic estimate}
	\label{subsec:equivalence_qe}

	The operator family $\Theta_t$ introduced in the following definition will be the main player in our quadratic estimate.

	\begin{definition}
		For every $t>0$, the operator $\Theta_t \colon \L^2(D) \oplus \L^q(\Gamma) \to \W^{1,2}(D)$ is defined by
		\begin{align}
			\Theta _t = t(1+t^2 \cL)^{-1} S^* B.
		\end{align}
	\end{definition}

	We show uniform bounds in $\L^2(D)$ for $\Theta_t$ in the following result.

	\begin{lemma}
		\label{lem:Theta_t_bounded}
		The operators $\Theta_t$ act as bounded operators \begin{align}
			\Theta_t \colon \L^2(D)^n \oplus \L^q(\Gamma) \to \L^2(D)
		\end{align}
		uniformly for $t>0$.
		Moreover, there are constants $C,c \in (0,\infty)$ such that, for all measurable sets $E,F \subseteq \R^n$, all $t > 0$ and all $G\in \L^2(D)^n$, $g \in \L^q(\Gamma)$ satisfying $\supp(G) \subseteq E$ and $\supp(g) \subseteq \overline{E}$, we have for $$u \coloneqq \Theta_t \begin{bmatrix} G \\ g \end{bmatrix}$$ the estimate
		\begin{align}
			\left( \iint_F |u|^2 \right)^\frac{1}{2} \leq C \e^{-c\frac{\dist(E,F)}{t}} \left[ \left( \iint_E |G|^2 \right)^\frac{1}{2} + \left( \iint_{\overline{E} \cap \Gamma} |g|^q \right)^\frac{1}{q} \right].
		\end{align}
	\end{lemma}

	\begin{proof}
		Since $L$ and $L^*$ are of the same type, Proposition~\ref{prop:ODE} yields that $t S (1 +t^2 L^*)^{-1}$ is $\L^2(D) \to \L^2(D)^n \oplus \L^q(\Gamma)$ bounded. Hence, by duality, $t(1 + t^2 \cL)^{-1} S^*$ is $\L^2(D)^n \oplus \L^{q'}(\Gamma) \to \L^2(D)$ bounded. Consequently, the first claim of the lemma follows from the mapping property of $B$ in~\eqref{eq:B}.

		Since $B$ is local and Proposition~\ref{prop:ODE} moreover yields off-diagonal estimates for $tS (1 + t^2 L)^{-1}$ with respect to the same norms as above, we deduce off-diagonal estimates for $\Theta_t$ by the same argument.
	\end{proof}

	It is well-known that Theorem~\ref{thm:main} is equivalent to the validity of the quadratic estimates
	\begin{align}
		\label{eq:qe1}
		\begin{aligned}
			\int_0^\infty \| tL(1+t^2 L)^{-1} u \|_{\L^2(D)}^2 \ddt &\lesssim \| u \|_{\W^{1,2}(D)}^2, \\
			\int_0^\infty \| tL^*(1+t^2 L^*)^{-1} u \|_{\L^2(D)}^2 \ddt &\lesssim \| u \|_{\W^{1,2}(D)}^2,
		\end{aligned}
	\end{align}
	see for instance~\cite[Prop.~12.7]{ISEM}. Note that $L^*$ is an elliptic operator with Robin boundary conditions associated with the coefficient matrix $B^*$. Since the ellipticity assumptions are invariant under conjugation, it suffices to verify only the first quadratic estimate in~\eqref{eq:qe1}.

	Let $u\in \W^{1,2}(D)$. Use the decomposition~\eqref{eq:factorization} to rewrite
	\begin{align}
		tL(1+t^2 L)^{-1} u = t (1+t^2 \cL)^{-1} \cL u = t (1+t^2 \cL)^{-1} S^* B S u = \Theta_t S u.
	\end{align}
	Therefore, the quadratic estimate for $L$ in~\eqref{eq:qe1} and thus Theorem~\ref{thm:main}, are equivalent to the quadratic estimate
	\begin{align}
		\label{eq:qe2}
		\tag{QE}
		\int_0^\infty \| \Theta_t Su \|_{\L^2(D)}^2 \ddt \lesssim \| u \|_{\W^{1,2}(D)}^2, \qquad u \in \W^{1,2}(D).
	\end{align}

	\subsection{Smoothing of the quadratic estimate}
	\label{subsec:smoothing}

	As explained in Remark~\ref{rem:weak_Laplacian}, we are going to introduce smoothing in terms of $S^*S$ into the quadratic estimate~\eqref{eq:qe2}.

	\begin{definition}
		For $t>0$, define the smoothing operator
		\begin{align}
			P_t \coloneqq (1 + t^2 S^*S)^{-1} \in \cL(\L^2(D)).
		\end{align}
	\end{definition}

	By simple algebraic manipulations with the smoothing operators, we find the following reduction.

	\begin{lemma}[Introduce smoothing]
		\label{lem:smoothing}
		There holds the quadratic estimate
		\begin{align}
			\int_0^\infty \| \Theta_t S (1-P_t)u \|_{\L^2(D)}^2 \ddt \lesssim \| u \|_{\W^{1,2}(D)}^2, \quad u \in \W^{1,2}(D).
		\end{align}
	\end{lemma}

	\begin{proof}
		Observe the algebraic identity
		\begin{align}
			\Theta_t S(1-P_t) u = t (1 + t^2 \cL)^{-1} S^* B S t^2 S^* S P_t u = t^2 \cL (1 + t^2 \cL)^{-1} t (S^* S)^\frac{1}{2} P_t (S^* S)^\frac{1}{2} u,
		\end{align}
		where we used that fractional powers of $S^* S$ commute with $P_t$. Note that this argument uses the square root property of $S^*S$ in the self-adjoint case.
		Hence, using maximal accretivity, we find
		\begin{align}
			\int_0^\infty \| \Theta_t S (1-P_t)u \|_{\L^2(D)}^2 \ddt \leq 2 \int_0^\infty \| (t^2 S^* S)^\frac{1}{2} P_t (S^* S)^\frac{1}{2} u \|_{\L^2(D)}^2 \ddt.
		\end{align}
		The right-hand side is a square-function estimate for the self-adjoint operator $S^*S$ applied to $(S^* S)^\frac{1}{2} u$, so the spectral theorem gives
		\begin{align}
			\int_0^\infty \| \Theta_t S (1-P_t) u \|_{\L^2(D)}^2 \ddt \lesssim \| (S^* S)^\frac{1}{2} u \|_{\L^2(D)}^2.
		\end{align}
		Using the square root property for the self-adjoint operator $S^* S$ another time concludes the proof.
	\end{proof}

	\subsection{Reduction to small scales in the interior}
	\label{subsec:finite_time_volume}

	So far, we have reduced~\eqref{eq:qe2} to the quadratic estimate
	\begin{align}
		\label{eq:qe3}
		\int_0^\infty \| \Theta_t S P_t u \|_{\L^2(D)}^2 \ddt \lesssim \| u \|_{\W^{1,2}(D)}^2, \quad u \in \W^{1,2}(D).
	\end{align}
	Expanding the operator in the quadratic estimate yields
	\begin{align}
		\label{eq:volume_bdd_decomp}
		\Theta_t S P_t = -t (1 + t^2 \cL)^{-1} \Div A \nabla P_t + t (1 + t^2 \cL)^{-1} \Tr^* a \Tr P_t.
	\end{align}
	To ease notation, put
	\begin{align}
		\wt \Theta_t \coloneqq -t (1 + t^2 \cL)^{-1} \Div A \colon \L^2(D) \to \W^{1,2}(D).
	\end{align}
	Note that $\wt \Theta_t$ inherits off-diagonal estimates from $\Theta_t$.
	By definition, we have
	\begin{align}
		\Theta_t S P_t = \wt \Theta_t \nabla P_t + t (1 + t^2 \cL)^{-1} \Tr^* b \Tr P_t.
	\end{align}
	In this subsection, we will on the one hand treat the large scales $t \geq 1$ in~\eqref{eq:qe3}, and on the other hand reduce matters to the \enquote{interior component}, meaning the term $\wt \Theta_t$ acting only in $D$ and not on $\Gamma$.

	\begin{lemma}[Reduction to small scales]
		\label{lem:finite_scales}
		There holds the quadratic estimate
		\begin{align}
			\int_1^\infty \| \Theta_t S P_t u \|_{\L^2(D)}^2 \ddt \lesssim \| u \|_{\W^{1,2}(D)}^2, \quad u \in \W^{1,2}(D).
		\end{align}
	\end{lemma}

	We remark that the proof implicitly exploits that $D$ is bounded.

	\begin{proof}
		Using $\L^2(D) \to \L^2(D)$ boundedness of $\Theta_t$, followed by boundedness of the trace operator, calculate for $t>0$:
		\begin{align}
			\| \Theta_t S P_t u \|_{\L^2(D)} &\lesssim \| \nabla P_t u \|_{\L^2(D)^n} + \| \Tr(P_t u) \|_{\L^q(\Gamma)} \lesssim \| P_t u \|_{\W^{1,2}(D)}.
		\end{align}
		Specializing to $t \geq 1$ and using the embedding $\H^{1+\alpha,2}(D) \subseteq \W^{1,2}(D)$, conclude
		\begin{align}
			\| \Theta_t S P_t u \|_{\L^2(D)}^2 \lesssim t^{2\alpha} \| P_t u \|_{\H^{1+\alpha,2}(D)}^2,
		\end{align}
		where $\alpha$ was introduced in Proposition~\ref{prop:coercivity}.
		Now, using Proposition~\ref{prop:coercivity}, we obtain
		\begin{align}
			\| \Theta_t S P_t u \|_{\L^2(D)}^2 \lesssim t^{2\alpha} \| (S^*S)^{\frac{1+\alpha}{2}} P_t u \|_{\L^2(D)}^2 = \| (t^2 S^*S)^{\frac{\alpha}{2}} P_t (S^*S)^{\frac{1}{2}} u \|_{\L^2(D)}^2.
		\end{align}
		Eventually, quadratic estimates and the square root property for $S^*S$ yield
		\begin{align}
				\int_1^\infty \| \Theta_t S P_t u \|_{\L^2(D)}^2 \ddt \lesssim \int_0^\infty \| (t^2 S^*S)^{\frac{\alpha}{2}} P_t (S^*S)^{\frac{1}{2}} u \|_{\L^2(D)}^2 \ddt
				&\lesssim \| (S^*S)^{\frac{1}{2}} u \|_{\L^2(D)}^2 \\
				&\approx \| u \|_{\W^{1,2}(D)}^2
		\end{align}
		as desired.
	\end{proof}

	By virtue of the mapping properties of the trace operator, the quadratic estimate on the boundary is non-singular and can hence be controlled easily.

	\begin{lemma}[Reduction to the interior]
		\label{lem:volume}
		There holds the quadratic estimate
		\begin{align}
			\int_0^1 \| t (1 + t^2 \cL)^{-1} \Tr^* a \Tr P_t u \|_{\L^2(D)}^2 \ddt \lesssim \| u \|_{\W^{1,2}(D)}^2, \quad u \in \W^{1,2}(D).
		\end{align}
	\end{lemma}

	\begin{proof}
		By Proposition~\ref{prop:ODE}, the resolvent family $(1 + t^2 L^*)^{-1}$ is  $\L^2(D) \to \L^2(D)$ bounded with norm controlled by $1$, and $\L^2(D) \to \H^1(D)$ bounded with norm controlled by $t^{-1}$. Hence, by interpolation, it is $\L^2(D) \to \H^{s,2}(D)$ bounded with norm controlled by $t^{-s}$, where $s \in (0,1)$ was introduced below Lemma~\ref{lem:parameter}. Thus, using boundedness of the trace operator as stated in Lemma~\ref{lem:parameter}, $\Tr(1 + t^2 L^*)^{-1}$ is $\L^2(D) \to \L^q(\Gamma)$ bounded with norm controlled by $t^{-s}$. Consequently, a duality argument, followed by the mapping properties of the coefficient function $a$ and the trace operator, give
		\begin{align}
			\| t (1 + t^2 \cL)^{-1} \Tr^* b \Tr P_t u \|_{\L^2(D)}^2 \lesssim t^{2(1-s)} \| b \Tr P_t u \|_{\L^{q'}(\Gamma)}^2 \lesssim t^{2(1-s)} \| P_t u \|_{\H^{s,2}(D)}^2.
		\end{align}
		Next, we use the embedding $\W^{1,2}(D) \subseteq \H^{s,2}(D)$ and two times the square root property for $S^*S$, to give
		\begin{align}
			\| P_t u \|_{\H^{s,2}(D)} \lesssim \| P_t u \|_{\W^{1,2}(D)} = \| (S^*S)^\frac{1}{2} P_t u \|_{\L^2(D)} = \| P_t (S^*S)^\frac{1}{2} u \|_{\L^2(D)} \lesssim \| u \|_{\W^{1,2}(D)}.
		\end{align}
		Thus,
		\begin{align}
			\| t (1 + t^2 \cL)^{-1} \Tr^* b \Tr P_t u \|_{\L^2(D)}^2 \lesssim t^{2(1-s)} \| u \|_{\W^{1,2}(D)}^2,
		\end{align}
		and therefore
		\begin{align}
			\int_0^1 \| t (1 + t^2 \cL)^{-1} \Tr^* b \Tr P_t u \|_{\L^2(D)}^2 \ddt \lesssim \int_0^1 t^{2(1-s)} \| u \|_{\W^{1,2}(D)}^2 \ddt
			\approx \| u \|_{\W^{1,2}(D)}^2,
		\end{align}
		where we use $1-s > 0$ in the last step.
		This concludes the proof.
	\end{proof}

	\subsection{Dyadic structures and dyadic averaging operator}
	\label{subsec:dyadic}

	We introduce a substitute in $D$ for the dyadic cubes in Euclidean space, called \enquote{dyadic structure}, and present basic properties of it. Then, we define the dyadic averaging operator induced by the dyadic structure, and present the Carleson embedding result for it.

	The construction of dyadic structures goes back to~\cite{ChristDyadic}. For our formulation here, we combine~\cite[Thm.~7.4]{Laplace-Extrapolation} with the boundary estimate~\cite[Lem.~5.2]{BEH} and the rescaling argument from~\cite[Prop. 2.12]{DeltaArg}.
	Write $\Q(x,\ell)$ for the Euclidean cube with center $x\in \R^n$ and sidelength $\ell > 0$ parallel to the coordinate axes.

	\begin{proposition}[Dyadic structure]
		\label{prop:cubes}
		There exists a collection $\{Q^k_i \colon k \in \N_0, i \in I_k\}$ of open
		sets, where $I_k$ are countable index sets, and constants $a_0$, $a_1$, $\eta$, $C > 0$ such that the following hold:
		\begin{enumerate}
			\item For each $k \in \N_0$, it holds $|D \setminus \cup_{i \in I_k} Q^k_i| = 0$.
			\item If $\ell \geq k$, then, for each $i \in I_k$ and each $j \in I_\ell$, either $Q^\ell_j \subseteq Q^k_i$ or $Q^\ell_j \cap Q^k_i = \emptyset$.
			\item If $\ell \leq k$, then, for each $i \in I_k$, there is a unique $j \in I_\ell$ such that $Q^k_i \subseteq Q^\ell_j$.
			\item For each $Q^k_i$, $k \in \N_0$, $i \in I_k$, there exists $z^k_i \in D$ such that $$\Q(z^k_i, a_0 2^{-k}) \cap D \subseteq Q^k_i \subseteq \Q(z^k_i, a_1 2^{-k}) \cap D.$$
			\item If $k \in \N_0$, $i \in I_k$, and $t>0$, then
			\begin{align}
				\bigl|\{x \in Q^k_i \colon \dist(x, \R^n \setminus Q^k_i) \leq t 2^{-k}\}\bigr| \leq Ct^{\eta}|Q^k_i|.
			\end{align}
		\end{enumerate}
	\end{proposition}

	Note that all properties except (v) work for $D$ merely interior thick. For (v), a mild form of boundary regularity, called \emph{porosity}, is needed. Locally uniform domains have a porous boundary.

	\begin{definition}[Dyadic cubes]
		\label{def:dyadic_cubes}
		In the setting of Proposition~\ref{prop:cubes}, put $\Box_{2^{-k}} \coloneqq \{Q^k_i \colon i \in I_k \}$. Members of this collection are called \emph{dyadic cubes} of \emph{generation $2^{-k}$}. For $t \in (2^{-(k+1)},2^{-k}]$, define $\Box_{t} \coloneqq \Box_{2^{-k}}$. In this case, the \emph{sidelength} of $Q \in \Box_{t}$ is $\ell(Q) \coloneqq 2^{-k}$. Finally, the collection $\Box$ of all (local) \emph{dyadic cubes} is given by $\bigcup_k \Box_{2^{-k}}$.
	\end{definition}

	Using metric arguments, a doubling operation on cubes can be defined.

	\begin{definition}[Double cube]
		If $Q \in \Box$, then
		define the \emph{doubled cube} $2Q$ via $2Q \coloneqq \{ x \in D \colon \dist(x, Q) < \tfrac{1}{2} \ell(Q) \}$.
	\end{definition}

	Using property~(iv) of dyadic cubes, interior thickness of $D$, and elementary geometric considerations, we find $|2Q| \lesssim |Q|$.

	\begin{definition}
		For $t \in (0,1]$, define the \emph{dyadic averaging operator} $\AvOp_t$ on $\L^2(D)^n$ by
		\begin{align}
			\AvOp_t(F)(x) \coloneqq \fiint_{Q_t} F,
		\end{align}
		where $Q_t$ is the unique set in $\Box_t$ containing $x$, meaningful for almost every $x \in D$.
	\end{definition}

	In order to formulate a boundedness result for the dyadic averaging operator below, we need to introduce the notation of a Carleson measure first.

	\begin{definition}
		Call a Borel measure $\mu$ on $D \times (0,1]$ a \emph{Carleson measure} if there exists a constant $C$ such that
		\begin{align}
			\label{eq:def carleson measure}
			\mu(R(Q)) \leq C |Q|, \quad Q \in \Box.
		\end{align}
		Here, $R(Q)$ denotes the Carleson box over $Q$ and is given by $R(Q) \coloneqq (0,\ell(Q)] \times Q \subseteq (0,1] \times D$. The Carleson norm $\| \mu \|_\mathcal{C}$ of $\mu$ is the infimum over all constants $C$ for which~\eqref{eq:def carleson measure} holds.
	\end{definition}

	The following version of the Carleson embedding for the dyadic averaging operator in a dyadic structure can be found in~\cite[Thm 4.3]{Morris}.

	\begin{proposition}[Carleson embedding]
		\label{prop:carleson_embedding}
		If $\mu$ is a Carleson measure on $D \times (0,1]$ with Carleson norm $\| \sigma \|_{\mathcal{C}}$, then
		\begin{align}
			\iiint_{D \times (0,1]} |(\AvOp_t u)(x)|^2 \d \mu(x,t) \lesssim \| \sigma \|_{\mathcal{C}} \|u \|_{\L^2(D)}^2.
		\end{align}
	\end{proposition}

	\subsection{Smoothed principal part approximation}
	\label{subsec:smoothed_ppa}

	In the previous subsection, we have reduced~\eqref{eq:qe2} to the validity of the quadratic estimate
	\begin{align}
		\label{eq:qe4}
		\int_0^1 \| \wt \Theta_t \nabla P_t u \|_{\L^2(D)}^2 \ddt \lesssim \| u \|_{\W^{1,2}(D)}^2, \quad u \in \W^{1,2}(D).
	\end{align}
	Next, we reduce this quadratic estimate to its \enquote{smoothed principal part}. The latter is defined in the subsequent definition.

	Using off-diagonal estimates, the operator $\wt \Theta_t$ can be extended to a well-defined, linear operator $\L^\infty(D)^n \to \L^2_{\loc}(D)$ uniformly for $t > 0$, see~\cite{Kato-Square-Root-Proof}. The subsequent definition uses this extension.

	\begin{definition}[Principal part]
		Identifying the standard unit vectors $e_1,\dots,e_n \in \C^n$ with their respective constant functions on $D$, define for $t \in (0,1]$ the \emph{principal part} of $\wt \Theta_t$ by
		\begin{align}
			\wt \gamma_t \coloneqq (\wt \Theta_t(e_1), \dots, \wt \Theta_t(e_n)) \in \L^2_{\loc}(D)^n.
		\end{align}
	\end{definition}

	\begin{remark}
		For almost every $x\in D$, the principal part $\wt \gamma_t(x)$ is a linear functional on $\C^n$ acting as
		\begin{align}
			\wt \gamma_t(x) w = \wt \Theta_t(w\ind_D)(x), \quad w\in \C^n.
		\end{align}
	\end{remark}

	The following bound is established as in~\cite{Kato-Square-Root-Proof}.

	\begin{lemma}
		\label{lem:ppa_uniform_bound}
		For all $t \in (0,1]$ and $F \in \L^2(D)^n$, we have
		\begin{align}
			\| \wt \gamma_t \AvOp_t F \|_{\L^2(D)} \lesssim \| F \|_{\L^2(D)^n}.
		\end{align}
	\end{lemma}

	As a further reduction step, for $0 < t \leq 1$ decompose
	\begin{align}
		\wt \Theta_t = (\wt \Theta_t - \wt \gamma_t \AvOp_t) + \wt \gamma_t \AvOp_t.
	\end{align}

	\begin{lemma}[Smooth principal part approximation]
		\label{lem:smooth_ppa_volume}
		There holds the quadratic estimate
		\begin{align}
			\int_0^1 \| (\wt \Theta_t - \wt \gamma_t \AvOp_t) \nabla P_t u \|_{\L^2(D)}^2 \ddt \lesssim \| u \|_{\W^{1,2}(D)}^2, \quad u \in \W^{1,2}(D).
		\end{align}
	\end{lemma}

	\begin{proof}
		We subdivide the proof into two steps.

		\textbf{Step 1}: we have $\| \wt \Theta_t - \wt \gamma_t \AvOp_t \|_{\cL(\H^{\theta,2}(D)^n, \L^2(D))} \lesssim t^\theta$ for $t \in (0,1]$ and $\theta \in [0,1]$.

		We argue by interpolation. To this end, let $t \in (0,1]$ and $F \in \W^{1,2}(D)^n$. We claim the bound
		\begin{align}
			\label{eq:smooth_ppa_1}
			\bigl\| (\wt \Theta_t - \wt \gamma_t \AvOp_t) F \bigr\|_{\L^2(D)} \lesssim t \| F \|_{\W^{1,2}(D)^n}.
		\end{align}
		Indeed, this follows from a standard argument using Poincaré's inequality, which uses in a crucial way the $\L^2(D)$ off-diagonal estimates for $\Theta_t$. For a version that applies to our geometric setting, see for instance~\cite[Lem.~7.8]{ISEM_Kato}.

		Using Lemma~\ref{lem:ppa_uniform_bound} and $\L^2(D)$ off-diagonal estimates for $\Theta_t$, we also obtain the bound
		\begin{align}
			\label{eq:smooth_ppa_2}
			\bigl\| (\wt \Theta_t - \wt \gamma_t \AvOp_t) F \bigr\|_{\L^2(D)} \lesssim \| F \|_{\L^2(D)^n}, \quad F \in \L^2(D)^n.
		\end{align}
		Then, interpolation between the bounds~\eqref{eq:smooth_ppa_1} and~\eqref{eq:smooth_ppa_2} yields the claim of this step.

		\textbf{Step 2}: conclusion using coercivity of $(S^*S)^\frac{1+\alpha}{2}$.

		Plug the bound from Step~1 with $\theta \coloneqq \alpha$ into the quadratic estimate, where $\alpha$ was introduced in Proposition~\ref{prop:coercivity}, to get
		\begin{align}
			\int_0^1 \| (\wt \Theta_t - \wt \gamma_t \AvOp_t) \nabla P_t u \|_{\L^2(D)}^2 \ddt \lesssim \int_0^1 t^{2\alpha} \| \nabla P_t u \|_{\H^{\alpha,2}(D)^n}^2 \ddt.
		\end{align}
		Using the mapping properties of $\nabla$ in the scale of Bessel potential spaces, we have $\| \nabla P_t u \|_{\H^{\alpha,2}(D)^n} \lesssim \| P_t u \|_{\H^{1+\alpha,2}(D)}$.
		Hence, Proposition~\ref{prop:coercivity} with $v \coloneqq P_t u = (1 + t^2 S^*S)^{-1} u \in \dom(S^*S)$ leads to
		\begin{align}
			\label{eq:qe5}
			\int_0^1 \| (\wt \Theta_t - \wt \gamma_t \AvOp_t) \nabla P_t u \|_{\L^2(D)}^2 \ddt &\lesssim \int_0^1 t^{2\alpha} \| (S^*S)^\frac{1+\alpha}{2} (1 + t^2 S^*S)^{-1} u \|_{\L^2(D)}^2 \ddt \\
			&= \int_0^1 \| (t^2 S^*S)^\frac{\alpha}{2} (1 + t^2 S^*S)^{-1} (S^*S)^\frac{1}{2} u \|_{\L^2(D)}^2 \ddt.
		\end{align}
		The expression on the right-hand side is dominated by a quadratic estimate for the operator $S^*S$, and is hence controlled by $\| (S^*S)^\frac{1}{2} u \|_{\L^2(D)}^2$.
		Finally, use the Kato square root property for the self-adjoint operator $S^*S$ to conclude.
	\end{proof}

	\subsection{Removing the smoothing operator}
	\label{subsec:remove_smoothing}

	So far, we have reduced~\eqref{eq:qe2} to the quadratic estimate
	\begin{align}
		\int_0^1 \| \wt \gamma_t \AvOp_t \nabla P_t u \|_{\L^2(D)}^2 \ddt \lesssim \| u \|_{\W^{1,2}(D)}^2, \quad u\in \W^{1,2}(D).
	\end{align}
	Next, we split
	\begin{align}
		\wt \gamma_t \AvOp_t \nabla P_t = \wt \gamma_t \AvOp_t \nabla (P_t - 1) + \wt \gamma_t \AvOp_t \nabla
	\end{align}
	to remove the smoothing operator $P_t$ from the quadratic estimate.

	\begin{lemma}[Removing the smoothing operator]
		\label{lem:remove_smoothing_operator}
		There holds the quadratic estimate
		\begin{align}
			\label{eq:qe7}
			\int_0^1 \| \wt \gamma_t \AvOp_t \nabla (P_t - 1) u \|_{\L^2(D)}^2 \ddt \lesssim \| u \|_{\W^{1,2}(D)}^2, \quad u \in \W^{1,2}(D).
		\end{align}
	\end{lemma}

	\begin{proof}
		By $\wt \gamma_t \AvOp_t = (\wt \gamma_t \AvOp_t) \AvOp_t$ and Lemma~\ref{lem:ppa_uniform_bound}, we can remove $\wt \gamma_t$ from the quadratic estimate.
		We start out with the reproducing identity
		\begin{align}
			\AvOp_t \nabla (P_t - 1) u = \int_0^\infty \AvOp_t \nabla (P_t - 1) s^2 S^*S (1 + s^2 S^*S)^{-2} u \dds, \quad u \in \W^{1,2}(D).
		\end{align}
		The proof relies on the usual Schur-type argument: provided
		\begin{align}
			\label{eq:schur}
			\| \AvOp_t s \nabla (P_t - 1) P_s \|_{\cL(\L^2(D))} \lesssim \eta(t/s), \quad s \in (0,\infty), t\in (0,1),
		\end{align}
		holds for a function $\eta \in \L^1((0,\infty), \ddt)$, then the claim follows from Young's convolution inequality
		\begin{align}
				&\int_0^1 \| \AvOp_t \nabla (P_t - 1) u \|_{\L^2(D)}^2 \ddt \\
				\lesssim{} &\left( \int_0^\infty |\eta(t)| \ddt \right)^2 \int_0^\infty \| (s^2 S^*S)^\frac{1}{2} (1 + s^2 S^*S)^{-1} (S^*S)^\frac{1}{2} u \|_{\L^2(D)}^2 \dds \\
				\lesssim{} &\| (S^*S)^\frac{1}{2} u \|_{\L^2(D)}^2 \\
				\approx{} &\| u \|_{\W^{1,2}(D)}^2.
		\end{align}
		Here, we used again the Kato square root property of $S^*S$ in the last step.

		So, it remains to establish~\eqref{eq:schur}. Let $f\in \L^2(D)$. First, consider $t \leq s$. Note that \begin{align}
			s \nabla (P_t - 1) P_s = \frac{t}{s} t\nabla P_t (P_s - 1).
		\end{align}
		Hence, using that the averaging operator is a contraction, calculate
		\begin{align}
			\| \AvOp_t s \nabla (P_t - 1) P_s f \|_{\L^2(D)} \leq \| s \nabla (P_t - 1) P_s f \|_{\L^2(D)^n}
			&= \frac{t}{s} \| t\nabla P_t (P_s - 1) f \|_{\L^2(D)^n} \\
			&\lesssim \frac{t}{s} \| f \|_{\L^2(D)}.
		\end{align}
		We also used uniform bounds for $(t \nabla P_t)_{t>0}$ here, which follow from ellipticity.
		Now, consider $s \leq t$. Split
		\begin{align}
			\label{eq:schur2}
			\| \AvOp_t s \nabla (P_t - 1) P_s f \|_{\L^2(D)} &\leq \| \AvOp_t s \nabla P_t P_s f \|_{\L^2(D)} + \| \AvOp_t s \nabla P_s f \|_{\L^2(D)} \\
			&\leq \| s \nabla P_t P_s f \|_{\L^2(D)} + \| \AvOp_t s \nabla P_s f \|_{\L^2(D)},
		\end{align}
		where we used again in the last step that the averaging operator is a contraction. For the first term on the right-hand side, we readily find
		\begin{align}
			\| s \nabla P_t P_s f \|_{\L^2(D)} \leq \frac{s}{t} \| t \nabla P_t P_s f \|_{\L^2(D)} \lesssim \frac{s}{t} \| f \|_{\L^2(D)}.
		\end{align}
		For the second term on the right-hand side of~\eqref{eq:schur2}, we recall the following interpolation inequality, see~\cite[Lem.~7.11]{ISEM_Kato} for a proof in our geometric setting:
		\begin{align}
			\label{eq:interpolation_inequality}
			\left| \fint_Q \nabla u \right|^2 \lesssim \frac{1}{t^\eta} \left( \fint_Q |u|^2 \right)^\frac{\eta}{2} \left( \fint_Q |\nabla u|^2 \right)^{1 - \frac{\eta}{2}}, \quad t\in (0,1], Q \in \Box_t, u \in \W^{1,2}(D).
		\end{align}
		Here, the number $\eta$ was introduced in Proposition~\ref{prop:cubes}.
		Also, observing that $\AvOp_t$ is a projection, we note
		\begin{align}
			\| \wt \gamma_t \AvOp_t s \nabla P_s f \|_{\L^2(D)} = \| \bigl( \wt \gamma_t \AvOp_t \bigr) \AvOp_t s \nabla P_s f \|_{\L^2(D)} \lesssim \| \AvOp_t s \nabla P_s f \|_{\L^2(D)}.
		\end{align}
		Hence, using the interpolation inequality and Hölder's inequality for sequence spaces, calculate
		\begin{align}
			\| \AvOp_t s \nabla P_s f \|_{\L^2(D)}^2 &= \sum_{Q \in \Box_t} |Q| \left| \fint_Q s \nabla P_s f \right|^2 \\
			&\lesssim \frac{1}{t^\eta} \sum_{Q \in \Box_t} |Q| \left( \fint_Q |s P_s f|^2 \right)^\frac{\eta}{2} \left( \fint_Q |s \nabla P_s f|^2 \right)^{1-\frac{\eta}{2}} \\
			&= \left( \frac{s}{t} \right)^\eta \sum_{Q \in \Box_t} \left( \int_Q |P_s f|^2 \right)^\frac{\eta}{2} \left( \int_Q |s\nabla P_s f|^2 \right)^{1 - \frac{\eta}{2}} \\
			&\leq \left( \frac{s}{t} \right)^\eta \| P_s f \|_{\L^2(D)}^\eta \| s \nabla P_s \|_{\L^2(D)}^{2 - \eta} \\
			&\lesssim \left( \frac{s}{t} \right)^\eta \| f \|_{\L^2(D)}^2.
		\end{align}
		In summary, we have established~\eqref{eq:schur2} with $\eta(\tau) \coloneqq C \min(\tau, \tau^{-1} + \tau^{-\eta})$ for a suitable absolute constant $C$. This completes the proof.
	\end{proof}

	\subsection{Full principal part approximation and reduction to a Carleson measure property}
	\label{subsec:ppa_conclusion}

	Combining all reduction steps so far, we have reduced~\eqref{eq:qe2} to the quadratic estimate
	\begin{align}
		\label{eq:qe6}
		\int_0^1 \| \wt \gamma_t \AvOp_t \nabla u \|_{\L^2(D)}^2 \ddt \lesssim \| u \|_{\W^{1,2}(D)}^2, \quad u \in \W^{1,2}(D).
	\end{align}
	We can bound the left-hand side as follows:
	\begin{align}
		\int_0^1 \| \wt \gamma_t \AvOp_t \nabla u \|_{\L^2(D)}^2 \ddt \leq \int_0^1 \iint_D |\AvOp_t \nabla u|^2  |\wt \gamma_t|^2 \ddxt.
	\end{align}
	Hence,~\eqref{eq:qe6} follows from the Carleson embedding (Proposition~\ref{prop:carleson_embedding}) provided we can show that the measure $\mu(t,x) \coloneqq |\wt \gamma_t|^2 \, \tfrac{\mathrm{d} x \mathrm{d} t}{t}$ is a Carleson measure on $D$.

	We record the full principal part approximation in the interior for later use.

	\begin{proposition}[Principle part approximation]
		\label{prop:ppa}
		There holds the quadratic estimate
		\begin{align}
			\int_0^\infty \| (\wt\Theta_t - \wt \gamma_t \AvOp_t) \nabla u \|_{\L^2(D)}^2 \ddt \lesssim \| u \|_{\W^{1,2}(D)}^2, \quad u \in \W^{1,2}(D).
		\end{align}
	\end{proposition}

	\begin{proof}
		Write
		\begin{align}
			(\wt \Theta_t - \wt \gamma_t \AvOp_t) \nabla = (\wt \Theta_t - \wt \gamma_t \AvOp_t) \nabla P_t + \wt \Theta_t \nabla (1 - P_t) - \wt \gamma_t \AvOp_t \nabla (1 - P_t) \eqqcolon \I + \II + \III.
		\end{align}
		The quadratic estimate corresponding to $\I$ was established in Lemma~\ref{lem:smooth_ppa_volume}. As for $\II$, $\| \wt \Theta_t \nabla (1 - P_t) \|_{\cL(\L^2(D))} \leq \| \Theta_t S (1 - P_t) \|_{\cL(\L^2(D))}$, so we can invoke Lemma~\ref{lem:smoothing}.
		Finally, the quadratic estimate for $\III$ was just established in Lemma~\ref{lem:remove_smoothing_operator}.
	\end{proof}

	\section{The Carleson measure estimate}
	\label{sec:carleson}

	The goal of this section is to show that the measure $\mu \coloneqq \| \wt \gamma_t \|^2 \ddxt$ is a Carleson measure on $D$. Owing to the reductions performed in the last section, this will imply our main result (Theorem~\ref{thm:main}).

	The key step is the construction of Tb--type test functions in the subsequent proposition. Since $\wt \gamma_t$ is the principal part of $\wt \Theta_t$, an operator acting only in the interior, the argument follow closely the case with Neumann boundary conditions. Therefore, we will keep this section rather short and refer to the literature for arguments where the Neumann case applies verbatim.

	\begin{proposition}[Tb--type test functions]
		There is a constant $\eps_0 \in (0,1)$ such that, for all $0 < \eps < \eps_0$, all unit vectors $\xi \in \C^n$ and each cube $Q \in \Box$, there is a test function $b = b^{\xi,Q,\eps} \in \L^2(D)^n$ such that the following hold:
		\begin{enumerate}
			\item[$\mathrm{(a)}$] $\| b \|_{\L^2(D)}^2 \lesssim |Q|$,
			\item[$\mathrm{(b)}$] $\Re \left( \xi \cdot \fint_Q b \right) \geq 1$,
			\item[$\mathrm{(c)}$] $\iiint_{R(Q)} |\wt \gamma_t(x) (\AvOp_t b)(x)|^2 \ddxt \lesssim \frac{|Q|}{\eps^2}$.
		\end{enumerate}
	\end{proposition}

	\begin{proof}
		Let $\xi \in \C^n$ be a unit vector, $\eps \in (0,1)$ and $Q \in \Box$. Write $\ell$ for the sidelength of $Q$. Fix a cutoff function $\eta \in \Cont_0^\infty(2Q)$ satisfying $\eta \equiv 1$ on $Q$ and $\| \eta \|_\infty + \ell \| \nabla \eta \|_\infty \leq C$, where $C$ is a dimensional constant, and set $\Phi(x) \coloneqq \eta(x)(x-x_Q) \cdot \overline{\xi}$.
		The important properties of $\Phi$ are
		\begin{align}
			\label{eq:Phi_props}
			\Phi \in \W^{1,2}(D), \quad |\nabla \Phi(x)| \lesssim \ind_{2Q}(x), \quad \nabla \Phi = \overline{\xi} \text{ on } Q.
		\end{align}
		Now, define the test function $b = b^{\xi,Q,\eps}$ by
		\begin{align}
			b \coloneqq 2\nabla (1 + (\eps \ell)^2 L)^{-1} \Phi.
		\end{align}
		Observe the identity
		\begin{align}
			\label{eq:diff}
			\frac{1}{2} b - \nabla \Phi = -(\eps\ell)^2 \nabla \bigl(1 + (\eps \ell)^2 \cL \bigr)^{-1} \cL \Phi = (\eps\ell)^2 \nabla \bigl(1 + (\eps \ell)^2 \cL \bigr)^{-1} \Div(A \nabla \Phi).
		\end{align}

		To obtain~(a), use the second property in~\eqref{eq:Phi_props} and~\eqref{eq:diff} to give
		\begin{align}
			\| b \|_{\L^2(D)}^2 \lesssim \| b - 2\nabla \Phi \|_{\L^2(D)}^2 + \| \nabla \Phi \|_{\L^2(D)}^2 \lesssim \| A \nabla \Phi \|_{\L^2(D)}^2 + \| \nabla \Phi \|_{\L^2(D)}^2 \lesssim |Q|,
		\end{align}
		with implicit constants independent of $\eps$ and $\ell$. Note that we also used uniform bounds for the family $(t^2 \nabla (1 + t^2 \cL)^{-1} \Div )_{t>0}$ here, see~\cite[Proof of Lem.~2.2]{Kato-Square-Root-Proof}.

		For~(b), use the third property in~\eqref{eq:Phi_props} to obtain
		\begin{align}
			\fiint_Q b - 2\overline{\xi} \d x = \fint_Q b - 2\nabla \Phi \d x = 2(\eps\ell)^2 \fint_Q \nabla u,
		\end{align}
		where $u \coloneqq \bigl( 1 + (\eps \ell)^2 \cL \bigr)^{-1} \Div(A \Phi)$. By uniform boundedness of$(t^2 \nabla (1 + t^2 \cL)^{-1} \Div )_{t>0}$ and $(t(1 + t^2 \cL)^{-1} \Div )_{t>0}$ in conjunction with the second property in~\eqref{eq:Phi_props}, we find
		\begin{align}
			\| u \|_{\L^2(D)} \lesssim \frac{1}{\eps \ell} |Q|^{\frac{1}{2}}, \quad \| \nabla u \|_{\L^2(D)} \lesssim |Q|^{\frac{1}{2}}.
		\end{align}
		Therefore,~\eqref{eq:interpolation_inequality} yields
		\begin{align}
			\label{eq:diff2}
			\left| \fiint_Q b - 2\overline{\xi} \d x \right|^2 &\lesssim \eps^4 \ell^{4 - \eta} \left( \fiint_Q |u|^2 \right)^\frac{\eta}{2} \left( \fiint_Q |\nabla u|^2 \right)^{1 - \frac{\eta}{2}} \lesssim \eps^\eta.
		\end{align}
		Consequently, since $\xi$ is a unit vector,
		\begin{align}
			\Re \left( \xi \cdot \fiint_Q b \d x \right) = 2 + \Re \left( \xi \cdot \fiint_Q b - 2\overline{\xi} \d x \right) \geq 2 - \sqrt{C \eps^\eta},
		\end{align}
		where $C$ is the implicit constant in~\eqref{eq:diff2}. Taking $\eps$ small enough so that $\sqrt{C \eps^\eta} \leq 1$ concludes the proof of~(b).

		As for (c), use Proposition~\ref{prop:ppa} with $F = b$ to give
		\begin{align}
			\frac{1}{2} \iiint_{R(Q)} |\wt \gamma_t(x) (\AvOp_t b)(x)|^2 \ddxt &\leq \int_0^\ell \Bigl( \| (\wt \gamma_t \AvOp_t - \wt \Theta_t) b \|_{\L^2(D)}^2 + \| \wt \Theta_t b \|_{\L^2(D)}^2 \Bigr) \ddt \\
			&\leq C \| b \|_{\L^2(D)}^2 + \int_0^\ell \| \wt \Theta_t b \|_{\L^2(D)}^2 \ddt \\
			&\leq C |Q| + \int_0^\ell \| \wt \Theta_t b \|_{\L^2(D)}^2 \ddt,
		\end{align}
		where we used (a) in the last step. To control the last term, expand
		\begin{align}
			\wt \Theta_t b &= -2t(1 + t^2 \cL)^{-1} \Div\bigl(A\nabla (1 + \eps^2 \ell^2 L)^{-1} \Phi \bigr) \\
			&= 2t(1 + t^2 \cL)^{-1} \cL (1 + \eps^2 \ell^2 L)^{-1} \Phi \\
			&= -2t(1 + t^2 L)^{-1} \bigl( (1 + \eps^2 \ell^2 \cL)^{-1} \Div \bigr) A \nabla \Phi.
		\end{align}
		Therefore, we obtain
		\begin{align}
			\| \wt \Theta_t b \|_{\L^2(D)}^2 \lesssim \frac{t^2}{\eps^2 \ell^2} \| A \nabla \Phi \|_{\L^2(D)}^2 \lesssim \frac{t^2}{\eps^2 \ell^2} |Q|,
		\end{align}
		where the bound $\| \nabla \Phi \|_{\L^2(D)}^2 \lesssim |Q|$ used in the last step follows from~\eqref{eq:Phi_props}.
		Hence, we find
		\begin{align}
			\int_0^\ell \| \wt \Theta_t b \|_{\L^2(D)}^2 \ddt \leq \frac{C}{\eps^2} |Q|
		\end{align}
		as desired.
	\end{proof}

	Having the test functions at hand, the Carleson measure estimate follows from the usual procedure using sectorial decomposition of $\C^n$, a stopping time argument, and the John--Nirenberg lemma for Carleson measures as in~\cite{Kato-Square-Root-Proof}. For a presentation with the same choice of test functions as above, see~\cite[Prop.~14.5, Lem.~14.8, Lem.~14.10]{ISEM}.

	\begin{proposition}[Carleson measure]
		\label{prop:carleson_measure}
		The measure $\mu = |\wt \gamma_t(x)|^2 \ddxt$ is a Carleson measure on $D$.
	\end{proposition}

	By assembling all reduction steps, we conclude Theorem~\ref{thm:main}.

	\begin{proof}[Proof of Theorem~\ref{thm:main}]
		As seen in Section~\ref{subsec:equivalence_qe}, Theorem~\ref{thm:main} follows from the quadratic estimate
		\begin{align}
			\int_0^\infty \| \Theta_t Su \|_{\L^2(D)}^2 \ddt \lesssim \| u \|_{\W^{1,2}(D)}^2, \qquad u \in \W^{1,2}(D).
		\end{align}
		Let $u \in \W^{1,2}(D)$. Using the reductions from Section~\ref{sec:reduction} (Lemma~\ref{lem:smoothing}, Lemma~\ref{lem:finite_scales}, Lemma~\ref{lem:volume}, Lemma~\ref{lem:smooth_ppa_volume}, and Lemma~\ref{lem:remove_smoothing_operator}) and Proposition~\ref{prop:carleson_measure} in conjunction with Proposition~\ref{prop:carleson_embedding}, calculate
		\begin{align}
			\int_0^\infty \| \Theta_t Su \|_{\L^2(D)}^2 \ddt
			\lesssim{} \| u \|_{\W^{1,2}(D)}^2 + \int_0^1 \iint_D |(\AvOp_t \nabla u)(x)|^2 \d \mu(t,x)
			\lesssim{} \| u \|_{\W^{1,2}(D)}^2
		\end{align}
		as desired.
	\end{proof}

\end{document}